%% file: koZ.tex
\documentclass{amsart}
\usepackage[textsize=scriptsize]{todonotes}

\usepackage{etoolbox}
\usepackage[margin=1in,marginparwidth=0.8in]{geometry}
\input{preamble2.tex}

\usepackage{etoolbox}
\newtoggle{final}
\toggletrue{final}

\newcommand{\KGL}{\mathrm{KGL}}
\newcommand{\KO}{\mathrm{KO}}

\newcommand{\Cor}{\mathrm{Cor}}
\newcommand{\Vect}{\mathrm{Vect}}
\newcommand{\Bil}{\mathrm{Bil}}
\newcommand{\Alt}{\mathrm{Alt}}
\newcommand{\cof}{\mathrm{cof}}
\newcommand{\W}{\mathrm{W}}

\numberwithin{proposition}{section}
\numberwithin{equation}{section}

\iftoggle{final} {
\renewcommand{\todo}[1]{}
\newcommand{\NB}[1]{}
}{ % else

\newcommand{\NB}[1]{\todo[color=gray!40]{#1}}
}

\title{The very effective covers of $\KO$ and $\KGL$ over Dedekind schemes}
\date{\today}

\author{Tom Bachmann}
\address{Mathematisches Institut, LMU Munich, Munich, Germany}
\email{\tomemail}

\begin{document}

\maketitle

\begin{abstract}
We answer a question of Hoyois--Jelisiejew--Nardin--Yakerson regarding framed models of motivic connective $K$-theory spectra over Dedekind schemes.
That is, we show that the framed suspension spectrum of the presheaf of groupoids of vector bundles (respectively non-degenerate symmetric bilinear bundles) is the effective cover of $\KGL$ (respectively very effective cover of $\KO$).
One consequence is that, over any scheme, we obtain a spectral sequence from Spitzweck's motivic cohomology to homotopy algebraic $K$-theory; it is strongly convergent under mild assumptions.
\end{abstract}

\section{Statement of results}
Let $S$ be a scheme.
The category $\PSh_\Sigma(\Cor^\fr(S))$ of presheaves with framed transfers \cite[\S2.3]{EHKSY} is a motivic analog of the classical category of $\scr E_\infty$-monoids.
We have the \emph{framed suspension spectrum} functor \[ \Sigma^\infty_\fr: \PSh_\Sigma(\Cor^\fr(S)) \to \SH(S) \] which was constructed in \cite[Theorem 18]{hoyois2018localization}.
By analogy with the classical situation, one might expect that many interesting motivic spectra can be obtained as framed suspension spectra.
This is indeed the case; see \cite[\S1.1]{hoyois2021hermitian} for a summary.

This note concerns the following examples of the above idea.
One has framed presheaves \cite[\S6]{hoyois2021hermitian} \[ \Vect, \Bil \in \PSh_\Sigma(\Cor^\fr(S)) \] where $\Vect(X)$ is the groupoid of vector bundles on $X$ and $\Bil(X)$ is the groupoid of vector bundles with a non-degenerate symmetric bilinear form.
There exist Bott elements \[ \beta \in \pi_{2,1} \Sigma^\infty_\fr \Vect \quad\text{and}\quad \tilde\beta \in \pi_{8,4} \Sigma^\infty_\fr \Bil \] and canonical equivalences \cite[Proposition 5.1]{hoyois2020hilbert} \cite[Proposition 6.7]{hoyois2021hermitian} \[ (\Sigma^\infty_\fr \Vect)[\beta^{-1}] \wequi \KGL \quad\text{and}\quad (\Sigma^\infty_\fr \Bil)[\tilde\beta^{-1}] \wequi \KO. \]
Here $\KGL$ is the motivic spectrum representing homotopy algebraic $K$-theory and $\KO$ is the motivic spectrum representing homotopy hermitian $K$-theory.\footnote{As a notational convention for this introduction, whenever we mention $\KO$ we shall assume that $1/2 \in S$.}
Again by comparison with the classical situation, this suggests that $\Sigma^\infty_\fr \Vect$ and $\Sigma^\infty_\fr \Bil$ should be motivic analogs of \emph{connective} $K$-theory spectra.
Another way of producing ``connective'' versions is by passing to (very) effective covers \cite{voevodsky-slice-filtration,spitzweck2012motivic}.
It was proved in \cite{hoyois2021hermitian,hoyois2020hilbert} that these two notions of connective motivic $K$-theory spectra coincide, provided that $S$ is regular over a field.

Our main result is to extend this comparison to more general base schemes.
We denote by $H\Z$ Spitzweck's motivic cohomology spectrum \cite{spitzweck2012motivic} and by $H\W$ the periodic Witt cohomology spectrum \cite[Definition 4.6]{bachmann-etaZ}.
\begin{theorem} \label{thm:main}
Let $S$ be a scheme.
\begin{enumerate}
\item Suppose that $f_1(H\Z) = 0 \in \SH(S)$. The canonical map \[ \Sigma^\infty_\fr \Vect \to f_0 \KGL \in \SH(S) \] is an equivalence.
\item Suppose in addition that $1/2 \in S$ and $H\W_{\ge 2} = 0 \in \SH(S)$.
  The canonical map \[ \Sigma^\infty_\fr \Bil \to \tilde f_0 \KO \in \SH(S) \] is an equivalence.
\end{enumerate}
These assumptions are satisfied if $S$ is essentially smooth over a Dedekind scheme (containing $1/2$ in case (2)).
\end{theorem}

\begin{remark}
That the assumptions are satisfied for Dedekind schemes is proved in \cite[Proposition B.4]{bachmann-norms} for (1) and in \cite[Lemma 3.8]{bachmann-etaZ} for (2).
They in fact hold for all schemes; this will be recorded elsewhere.
\end{remark}

\begin{example}
Bott periodicity implies formally that $f_n \KGL \wequi \Sigma^{2n,n} f_0\KGL$ and $s_n(\KGL) \wequi \Sigma^{2n,n} f_0(\KGL)/\beta$.
Theorem \ref{thm:main}(1) implies that $f_0(\KGL)/\beta \wequi H\Z$ (see Lemma \ref{lemm:Vect/beta}).
Hence in this situation the slice filtration for $\KGL$ yields a convergent spectral sequence, with $E_2$-page given by (Spitzweck's) motivic cohomology.
\end{example}

\subsection*{Notation}
We use notation for standard motivic categories and spectra, as in \cite{bachmann-etaZ} and \cite{hoyois2021hermitian}.

\section{Proofs}
As a warm-up, we treat the case of $\KGL$.
Recall that the functor $\Sigma^\infty_\fr$ inverts group-completion.
The Bott element lifts to $\beta: (\P^1, \infty) \to \Vect^\gp$ \cite[\S5]{hoyois2020hilbert}.
We also have the rank map $\Vect^\gp \to \Z \in \PSh_\Sigma(\Cor^\fr(S))$.
The composite \[ (\P^1,\infty) \wedge \Vect^\gp \xrightarrow{\beta} \Vect^\gp \wedge \Vect^\gp \xrightarrow{m} \Vect^\gp \to \Z \] is null-homotopic after motivic localization, since $\Z$ is motivically local and truncated and $(\P^1,\infty) \stackrel{\mot}{\wequi} S^1 \wedge \Gm$.
\begin{lemma} \label{lemm:Vect/beta}
The induced map \[ (\Sigma^\infty_\fr \Vect)/\beta \to \Sigma^\infty_\fr \Z \wequi H\Z \] is an equivalence.
\end{lemma}
\begin{proof}
The equivalence $\Sigma^\infty_\fr \Z \wequi H\Z$ is \cite[Theorem 21]{hoyois2018localization}.
Since all terms are stable under base change \cite[proof of Lemma 7.5]{hoyois2021hermitian} \cite[Lemma 16]{hoyois2018localization}, we may assume that $S = \Spec(\Z)$.
Using \cite[Proposition B.3]{bachmann-norms} we further reduce to the case where $S$ is the spectrum of a perfect field.
In this case $\Sigma^\infty_\fr \Vect \wequi f_0 \KGL$ and so $(\Sigma^\infty_\fr \Vect)/\beta \wequi s_0 \KGL \wequi H\Z$ (see e.g. \cite[Proposition 2.7]{ananyevskiy2017very}).
\end{proof}

\begin{proof}[Proof of Theorem \ref{thm:main}(1)]
Note first that if $U \subset S$ is an open subscheme, and any of the assumptions of Theorem \ref{thm:main} holds for $S$, it also holds for $U$.
On the other hand, if one of the conclusions holds for all $U$ in an open cover, it holds for $S$.
It follows that we may assume that $S$ is qcqs, e.g. affine.

Since $f_1(H\Z) = 0$ we find (using Lemma \ref{lemm:Vect/beta}) that \[ \beta: \Sigma^\infty_\fr \Vect \to \Sigma^{-2,-1} \Sigma^\infty_\fr \Vect \] induces an equivalence on $f_i$ for $i \ge 0$.
It follows that in the directed system \[ \Sigma^\infty_\fr \Vect \xrightarrow{\beta} \Sigma^{-2,-1} \Sigma^\infty_\fr \Vect \xrightarrow{\beta} \Sigma^{-4,-2} \Sigma^\infty_\fr \Vect \xrightarrow{\beta} \dots \] all maps induce an equivalence on $f_0$.
Since the colimit is $\KGL$, $f_0$ commutes with colimits (here we use that $X$ is qcqs, via \cite[Proposition A.3(2)]{bachmann-norms}) and $\Sigma^\infty_\fr \Vect$ is effective (like any framed suspension spectrum), the result follows.
\end{proof}

The proof for $\KO$ is an elaboration on these ideas.
From now on we assume that $1/2 \in S$.
Recall from \cite[Definition 2.6, Lemma 2.7]{bachmann-etaZ} the motivic spectrum \[ \ul{k}^M \wequi (H\Z/2)/\tau \in \SH(S). \]
For the time being, assume $S$ is Dedekind.
Taking framed loops we obtain \[ \ul{k}_1^M := \Omega^\infty_\fr \Sigma^{1,1} \ul{k}^M \in \PSh_\Sigma(\Cor^\fr(S)). \]
\begin{lemma} \label{lemm:k1M}
Let $S$ be a Dedekind scheme, $1/2 \in S$.
\begin{enumerate}
\item We have $\ul{k}_1^M \wequi a_\Nis \tau_{\le 0} \Gm/2$, where $\Gm \in \PSh_\Sigma(\Cor^\fr(S))$ denotes the sheaf $\scr O^\times$ with its usual structure of transfers \cite[Example 2.4]{lecture-notes-mot-cohom}.
\item If $f: S' \to S$ is a morphism of Dedekind schemes then $f^* \ul{k}_1^M \stackrel{\mot}{\wequi} \ul{k}_1^M \in \PSh_\Sigma(\Cor^\fr(S'))$.
\item The canonical map $\Sigma^\infty_\fr \ul{k}_1^M \to \Sigma^{1,1} \ul{k}^M \in \SH(S)$ is an equivalence.
\end{enumerate}
\end{lemma}
For this and some of the following arguments, it will be helpful to recall that we have an embedding of $\Spc^\fr(S)^\gp$ into the stable category of spectral presheaves on $\Cor^\fr(S)$.
In particular, many fiber sequences in $\Spc^\fr(S)$ are cofiber sequences.
\begin{proof}
(1) Clear by construction since $H^1_\et(X, \mu_2) \wequi \scr O^\times(X)/2$ for $X$ affine.

(2) By (1) we have a cofiber sequence $\Sigma \mu_2 \to a_\Nis \Gm/2 \to \ul{k}_1^M \in \PSh_\Sigma(\Cor^\fr(S))$.
Since pullback of framed presheaves preserves cofiber sequences and commutes with forgetting transfers up to motivic equivalence \cite[Lemma 16]{hoyois2018localization} we reduce to the same assertion about $\Gm, \mu_2$, viewed as presheaves without transfers.
Since they are representable, the assertion is clear.

(3) Using \cite[Proposition B.3]{bachmann-norms}, (2) and \cite[Theorem 4.4]{bachmann-etaZ} we may assume that $S$ is the spectrum of a perfect field.
In this case $\Sigma^\infty_\fr \Omega^\infty_\fr \wequi \tilde f_0$ \cite[Theorem 3.5.14(i)]{EHKSY}, so we need only prove that $\Sigma^{1,1} \ul{k}_1^M$ is very effective.
But this is clear since we have the cofiber sequence $\Sigma^{1,0} H\Z/2 \xrightarrow{\tau} \Sigma^{1,1} H\Z/2 \to \Sigma^{1,1} \ul{k}_1^M$ and $H\Z/2$ is very effective.
\end{proof}

\begin{construction}
The assignment $V \mapsto (V \oplus V^*, \varphi_V)$ sending a vector bundle to its associated (hyperbolic) symmetric bilinear bundle upgrades to a morphism \[ \Vect \to \Bil \in \PSh_\Sigma(\Cor^\fr(S))^{BC_2}, \] where $\Vect$ carries the $C_2$-action coming from passing to dual bundles, and $\Bil$ carries the trivial $C_2$-action.
\end{construction}
\begin{proof}
Since the presheaves are $1$-truncated, all the required coherence data can be written down by hand.\NB{Better argument?}
\end{proof}

\begin{lemma}
Let $S$ be a Dedekind scheme containing $1/2$.
\begin{enumerate}
\item The map \[ (\Vect^\gp)_{hC_2} \to \Bil^\gp \] induces an isomorphism on $a_\Nis \pi_i$ for $i = 1,2$.
\item The homotopy orbits spectral sequence yields \[ a_\Nis \pi_0 (\Vect^\gp)_{hC_2} \wequi \Z, \] an exact sequence \[ 0 \to \ul{k}_1^M \to a_\Nis \pi_1 (\Vect^\gp)_{hC_2} \to \Z/2 \to 0 \] and a map \[ a_\Nis \pi_2 (\Vect^\gp)_{hC_2} \to \Z/2, \] all as presheaves with framed transfers.
\end{enumerate}
\end{lemma}
\begin{proof}
(1) This follows from the cofiber sequence $K_{hC_2} \to \mathrm{GW} \to L$ \cite[Theorem 7.6]{schlichting2016hermitian} using that $a_\Nis \pi_i L = 0$ unless $i \equiv 0 \pmod{4}$.

(2) The homotopy orbit spectral sequence just arises from the Postnikov filtration of $\Vect^\gp$ and the formation of homotopy orbits and hence is compatible with transfers.
Its $E_2$ page takes the form \[ H_i(C_2, a_\Nis \pi_j \Vect^\gp) \Rightarrow a_\Nis \pi_{i+j} (\Vect^\gp)_{hC_2}. \]
The form of the differentials of the spectral sequence implies that the terms $H_i(C_2, a_\Nis \pi_j \Vect^\gp)$ are permanent cycles for $i \le 1$, and survive to $E_\infty$ for $(i,j) = (0,0)$ and $(i,j) = (1,1)$.
One has $a_\Nis \pi_0 \Vect^\gp = \Z$ with the trivial action and $a_\Nis \pi_1 \Vect^\gp = \Gm$ \cite[Lemma III.1.4]{weibel-k-book} with the inversion action.
This already yields the first assertion.
A straightforward computation shows that \[ H_*(C_2, \Z) = \Z, \Z/2, 0, \Z/2, \dots \] and \[ H_*(C_2, \Gm) = \ul{k}_1^M, \mu_2, \ul{k}_1^M, \dots. \]
Since $H_2(C_2, \Z) = 0$, no differential can hit the $(i,j) = (0,1)$ spot either, yielding the second assertion.
Moreover this implies that $H_1(C_2, \Gm) = \mu_2$ is the bottom of the filtration of $\pi_2$.
It follows that there is a map $a_\Nis \pi_2 (\Vect^\gp)_{hC_2} \to A$, where $A$ is a quotient of $\mu_2$.
To prove that $A=\mu_2$ it suffices to check this on sections over a field\NB{really?}, in which case we can use the hermitian motivic spectral sequence of \cite{bachmann-very-effective}.\NB{Better argument?}
\end{proof}
We have $a_\Nis \pi_0 \Bil^\gp \wequi \ul{GW}$.
Thus we can form the following filtration of $\Bil^\gp$ refining the Postnikov filtration \[ \Bil^\gp \leftarrow F_1 \Bil^\gp \leftarrow F_2 \Bil^\gp\leftarrow F_3 \Bil^\gp\leftarrow F_4 \Bil^\gp \in \PSh_\Sigma(\Cor^\fr(S)) \] with subquotients given Nisnevich-locally by \begin{equation} \label{eq:subquotients} \ul{GW}, \Sigma \Z/2, \Sigma \ul{k}_1^M, \Sigma^2 \Z/2. \end{equation}

Recall also the framed presheaf $\Alt \in \PSh_\Sigma(\Cor^\fr(S))$ sending a scheme to the groupoid of vector bundles with a non-degenerate alternating form.
Tensoring with the canonical alternating (virtual) form $H(1) - h$ on $H\P^1$ (where $H(1)$ is the tautological rank $2$ alternating form on $H\P^1$, and $h$ is the standard alternating form on a trivial vector bundle of rank $2$) yields maps \[ \sigma_1: H\P^1 \wedge \Alt^\gp \to \Bil^\gp \quad\text{and}\quad \sigma_2: H\P^1 \wedge \Bil^\gp \to \Alt^\gp; \] by construction we have $\tilde \beta = \sigma_1 \sigma_2$ (recall that $H\P^1 \stackrel{\mot}{\wequi} S^{4,2}$).

\begin{lemma}
Let $S$ be a Dedekind scheme, $1/2 \in S$.
\begin{enumerate}
\item The composite \[ H\P^1 \wedge \Alt^\gp \xrightarrow{\sigma_1} \Bil^\gp \to \Bil^\gp/F_4\Bil^\gp \] is motivically null.
  The induced map \[ \Sigma^\infty_\fr \cof(\sigma_1) \to \Sigma^\infty_\fr \Bil^\gp/F_4\Bil^\gp \] is an equivalence.
\item The composite \[ H\P^1 \wedge \Bil^\gp \xrightarrow{\sigma_2} \Alt^\gp \xrightarrow{rk/2} \Z \] is motivically null.
  The induced map \[ \Sigma^\infty_\fr \cof(\sigma_2) \to \Sigma^\infty_\fr \Z \] is an equivalence.
\end{enumerate}
\end{lemma}
\begin{proof}
(1) Write $C$ for the cofiber \emph{computed in the category of spectral presheaves on $\Cor^\fr(S)$}.
Then $C$ admits a finite filtration, with subquotients corresponding to those in \eqref{eq:subquotients}.
Since each of those is the infinite loop space of a motivic spectrum, it follows that $C$ is in fact motivically local.
Consequently $C$ corresponds to $\Bil^\gp/F_4\Bil^\gp$ under the embedding into spectral presheaves.
These contortions tell us that there are \emph{fiber} sequences \[ F_{i+1}\Bil^\gp/F_4\Bil^\gp \to F_i\Bil^\gp/F_4\Bil^\gp \to F_i\Bil^\gp/F_{i+1}\Bil^\gp \] for $i < 4$.
Hence to prove that the composite is null, it suffices to prove that there are no maps from $\Sigma^{4,2} \Alt^\gp$ into the motivic localizations of the subquotients of the filtration given in \eqref{eq:subquotients}.
These motivic localizations are $\ul{GW}, L_\Nis K(\Z/2,1), L_\Nis K(k_1^M,1)$ and $L_\Nis K(\Z/2,2)$ (since they are motivically equivalent to the subquotients, and motivically local because they are infinite loop spaces of the motivic spectra $H\tilde \Z$\NB{Use defining fiber square.}, $\Sigma \ul{k}^M$, $\Sigma^{2,1} \ul{k}^M, \Sigma^2 \ul{k}^M$).
It suffices to prove that $\Omega^{4,2}$ of these subquotients vanishes, which is clear.
Next we claim that $\Sigma^\infty_\fr \Bil^\gp/F_4\Bil^\gp$ is stable under base change (among Dedekind schemes containing $1/2$).
Indeed the defining fiber sequences of $F_4\Bil^\gp$ are also cofiber sequences, and so $\Sigma^\infty_\fr \Bil^\gp/F_4\Bil^\gp$ is obtained by iterated extension from spectra stable under base change (see Lemma \ref{lemm:k1M}(2) for $\ul{k}_1^M$, \cite[proof of Lemma 7.5]{hoyois2021hermitian} for $\Bil$ and $\Alt$, and \cite[Lemma 16]{hoyois2018localization} for $\Z/2$).
To prove that the induced map is an equivalence we thus reduce as before to $S = \Spec(k)$, $k$ a perfect field of characteristic $\ne 2$.
In this case the result is a straightforward consequence of the hermitian motivic filtration of \cite{bachmann-very-effective}.

(2) The proof is essentially the same as for (1), but easier.
\end{proof}

We now arrive at the main result.
\begin{theorem}
Let $S$ be a scheme containing $1/2$ such that \[ f_1(H\Z) = 0 = H\W_{\ge 2} \in \SH(S). \]
The canonical maps \[ \Sigma^\infty_\fr \Bil \to \tilde f_0 \KO \quad\text{and}\quad \Sigma^\infty_\fr \Alt \to \tilde f_0 \Sigma^{4,2} \KO \] are equivalences.
\end{theorem}
\begin{proof}
As before we may assume that $S$ is qcqs.

We know that $\KO$ is the colimit of \[ \Sigma^\infty_\fr \Bil \xrightarrow{\sigma_2} \Sigma^{-4,-2} \Sigma^\infty_\fr \Alt \xrightarrow{\sigma_1} \Sigma^{-8,-4} \Bil \xrightarrow{\sigma_2} \cdots. \]
It is hence enough to prove that $\sigma_1: \Sigma^{-8n,-4n} \Sigma^\infty_\fr \Bil \to \Sigma^{-8n-4,-4n-2} \Sigma^\infty_\fr \Alt$ induces an equivalence on $\tilde f_0$ for every $n \ge 0$, and similarly for $\sigma_2$.
(Here we use that $S$ is qcqs, so that $\tilde f_0$ preserves filtered colimits.\NB{ref?})
Given a cofiber sequence $A \to B \to C$, in order to prove that $\tilde f_0 A \wequi \tilde f_0 B$, it suffices to show that $\Map(X, C) = *$ for every $X \in \SH(S)^\veff$, i.e. that $C \in \SH(S)^{\veff\perp}$.

Over $\Z[1/2]$, the cofiber of $\sigma_1$ has a finite filtration, with subquotients \[ \Sigma^{-4,-2} \Sigma^\infty_\fr \ul{GW}, \Sigma^{-3,-2} \Sigma^\infty_\fr \Z/2, \Sigma^{-3,-2} \Sigma^\infty_\fr \ul{k}_1^M, \Sigma^{-2,-2} \Sigma^\infty_\fr \Z/2, \] and the cofiber of $\sigma_2$ is $\Sigma^{-4,-2} \Sigma^\infty_\fr \Z$.
Using \cite[Corollary 22]{hoyois2018localization}, \cite[Theorem 7.3]{hoyois2021hermitian} and Lemma \ref{lemm:k1M}(3), we can identify the list of cofibers as \[ \Sigma^{-4,-2} H\tilde\Z, \Sigma^{-3,-2} H\Z/2, \Sigma^{-2,-1} \ul{k}^M, \Sigma^{-2,-2} H\Z/2, \Sigma^{-4,-2} H\Z. \]
These spectra are stable under arbitrary base change (essentially by definition), and hence for arbitrary $S$ the cofibers of $\sigma_1, \sigma_2$ are obtained as finite extensions, with cofibers in the above list.
To conclude the proof, it will thus suffice to show that all spectra in the above list are in $\SH(S)^{\veff\perp}$.

Note that if $E \in \SH(S)$ then $E \in \SH(S)^{\veff\perp}$ if and only if $\Omega^\infty E \wequi *$.
In particular this holds if $f_0 E = 0$.
This holds for $\Sigma^{m,n} H\Z$ as soon as $n<0$, by assumption.
Hence it also holds for $\Sigma^{m,n}H\Z/2$ in the same case ($f_0$ being a stable functor) and for \[ \Sigma^{m,n} \ul{k}^M \wequi \cof(\Sigma^{m,n-1} H\Z/2 \xrightarrow{\tau} \Sigma^{m,n} H\Z/2). \]
The only spectrum left in our list is $\Sigma^{-4,-2}H\tilde\Z$.
Using \cite[Definition 4.1]{bachmann-etaZ} we see now that $\Omega^\infty \Sigma^{-4,-2} H\tilde\Z \wequi \Omega^\infty \Sigma^{-4,-2} \ul{K}^W$, so we may treat the latter spectrum.
We have $\ul{K}^W/\eta \wequi \ul{k}^M$ \cite[Lemma 3.9]{bachmann-etaZ}, whence $\eta: \Sigma^{-4-n,-2-n} \ul{K}^W \to \Sigma^{-5-n,-3-n} \ul{K}^W$ induces an equivalence on $\Omega^\infty$.
Since $\Omega^\infty$ commutes with filtered colimits, we see that $\Sigma^{-4,-2}\ul{K}^W \in \SH(S)^{\veff\perp}$ if and only if $\Sigma^{-4,-2}\ul{K}^W[\eta^{-1}] \in \SH(S)^{\veff\perp}$.
This latter spectrum is the same as $\Sigma^{-2} H\W$ \cite[Lemma 3.9]{bachmann-etaZ}, and \[ \tilde f_0(\Sigma^{-2} H\W) \wequi \tilde f_0((\Sigma^{-2} H\W)_{\ge 0}) \wequi \tilde f_0(\Sigma^{-2} (H\W_{\ge 2})) = 0 \] by assumption.
\end{proof}

\bibliographystyle{plainc}
\bibliography{bibliography}
\end{document}

%% file: preamble2.tex
\usepackage{amsmath}
\usepackage{amssymb}
\usepackage{amsthm}
\usepackage{amscd}
\usepackage{enumerate}
\usepackage[pdfusetitle,unicode,hidelinks]{hyperref}
\usepackage{bbm}
\usepackage{etoolbox}

\usepackage[utf8]{inputenc}
\usepackage[T1]{fontenc}

\newcommand{\tomemail}{\href{mailto:tom.bachmann@zoho.com}{tom.bachmann@zoho.com}}

\newtheorem{proposition}{Proposition}

\newtheorem{lemma}[proposition]{Lemma}
\newtheorem{theorem}[proposition]{Theorem}

\newtheorem*{conjecture*}{Conjecture}
\newtheorem*{theorem*}{Theorem}
\newtheorem*{corollary*}{Corollary}
\newtheorem*{proposition*}{Proposition}
\newtheorem*{lemma*}{Lemma}
\theoremstyle{definition}

\newtheorem{construction}[proposition]{Construction}

\newtheorem*{definition*}{Definition}
\newtheorem*{construction*}{Construction}
\theoremstyle{remark}
\newtheorem{remark}[proposition]{Remark}
\newtheorem*{remark*}{Remark}

\newtheorem{example}[proposition]{Example}
\newtheorem*{example*}{Example}

\newcommand{\Z}{\mathbb{Z}}

\let\scr=\mathcal
\let\bb=\mathbb
\newcommand{\Gm}{{\mathbb{G}_m}}

\def\P{\bb P}

\newcommand{\veff}{{\text{veff}}}

\newcommand{\SH}{\mathcal{SH}}

\let\lim=\relax
\DeclareMathOperator*{\lim}{lim}
\def\Map{\mathrm{Map}}

\def\PSh{\mathcal{P}}

\def\Spc{\mathcal{S}\mathrm{pc}{}}

\newcommand{\Spec}{\mathrm{Spec}}
\newcommand{\gp}{\mathrm{gp}}

\newcommand{\wequi}{\simeq}

\DeclareRobustCommand{\ul}{\underline}

\def\Nis{\mathrm{Nis}}

\def\mot{\mathrm{mot}}
\newcommand{\et}{{\acute{e}t}}

\newcommand{\fr}{\mathrm{fr}}